\documentclass[12pt]{amsart}

\usepackage{graphicx}
\usepackage{amssymb}
\usepackage{amsmath}
\usepackage{breqn}
\usepackage{physics}
\usepackage{mathtools}
\usepackage[utf8x]{inputenc}

\newtheorem{theorem}{Theorem}[section]
\newtheorem{cor}[theorem]{Corollary}
\newtheorem{lemma}[theorem]{Lemma}
\newtheorem{proposition}[theorem]{Proposition}

\theoremstyle{definition}

\theoremstyle{remark}

\numberwithin{equation}{section}

\begin{document}

\title{Julia sets of random exponential maps}

\author{Krzysztof Lech}
\address{University of Warsaw, Faculty of Mathematics, Informatics and Mechanics}
\email{K.Lech@mimuw.edu.pl}
\thanks{Research  supported in part by the NCN grant 2014/13/B/ST1/04551}



\begin{abstract}
For a sequence $(\lambda_n)$ of positive real numbers we consider the exponential functions $f_{\lambda_n} (z) = \lambda_n e^z$ and the compositions $F_n = f_{\lambda_n} \circ f_{\lambda_{n-1}} \circ ... \circ f_{\lambda_1}$. For such a non-autonomous family we can define the Fatou and Julia sets analogously to the usual case of autonomous iteration. The aim of this document is to study how the Julia set depends on the sequence $(\lambda_n)$. Among other results, we prove the Julia set for a random sequence $\{\lambda_n \}$, chosen uniformly from a neighbourhood of $\frac{1}{e}$, is the whole plane with probability $1$. We also prove the Julia set for $\frac{1}{e} + \frac{1}{n^p}$ is the whole plane for $p < \frac{1}{2}$, and give an example of a sequence $\{\lambda_n \} $ for which the iterates of $0$ converge to infinity starting from any index, but the Fatou set is non-empty.
\end{abstract}

\maketitle

\section{Introduction}
We consider a sequence $(\lambda_n)$ of positive real numbers, bounded from above by some constant. Let us denote $f_{\lambda_n} (z) = \lambda_n e^z$ and let $F_n$ be the sequence of iterates $F_n = f_{\lambda_n} \circ f_{\lambda_{n-1}} \circ ... \circ f_{\lambda_1}$. Then the Fatou set of a sequence $F(\lambda_n)$ is the set of all $z \in \mathbb{C}$ for which $F_n$ is normal on a neighbourhood of $z$, that is from every sequence one can extract an almost uniformly convergent subsequence. The Julia set $J(\lambda_n)$ is the complement of the Fatou set. These are natural extensions of the same definitions for iterates of a fixed function, rather than a non-autonomous sequence.

For an autonomous sequence of iterates of $\lambda e^z$, where the $\lambda$ is fixed, it is well known that the Julia set is the whole plane for $\lambda > \frac{1}{e}$. In particular the case of $\lambda = 1$ is a famous result by Misiurewicz (\cite{MM}). For $\lambda = \frac{1}{e}$ this is not the case, as there is a nonempty Fatou set, namely the parabolic basin of the fixed point $1$. It should not be surprising then that $\frac{1}{e}$ plays a special role in our considerations.

Some work on non-autonomous iteration of the exponential function has been done in \cite{MU}, \cite{MUZ}, \cite{ZU}, \cite{ZU2}, with \cite{ZU} being especially important to us for yielding some key corollaries. Considerable work has also been done in non-autonomous iteration of the quadratic family, in \cite{RB}, \cite{BB}, \cite{QIU} among other papers. These along with \cite{Com} and \cite{FS} were a useful inspiration, even though they do not explicitly cover non polynomial dynamics. Some common themes can be found, especially in the techniques used for building pathological examples in non-autonomous iteration in general, whether transcendental or rational. 

It is worth pointing out some substantial differences between autonomous and non-autonomous complex dynamics. Most importantly in our setting there is no reasonable notion of a periodic point (since the function changes along iteration). Because of that some theory from autonomous dynamics either does not apply, or has to be proven using other methods. A lot of groundwork for this has been done in the aforementioned papers.

This document aims to expand on one of the results from \cite{ZU}, which answers the question of what is the Julia set when we separate the sequence $\lambda_n$ from $\frac{1}{e}$. In the second section we recall this theorem, and for the reader's convenience the original authors' proof is available in the appendix. Presenting this proof in its entirety is necessary for pointing out some corollaries arising from it, which were not explicitly stated by the authors.
The third section contains various generalizations of the aforementioned result, including randomizing $\lambda_n$ uniformly from an interval. We also consider sequences convergent to $\frac{1}{e}$, by taking a close look at $\lambda_n = \frac{1}{e} + \frac{1}{n^p}$ for various $p$.

The author would like to thank Anna Zdunik for suggesting this topic, and providing useful conversations throughout work on this project.

\section{The Julia set for $M > \lambda_n > \bar{\lambda} > \frac{1}{e}$ and consequences}
The case for which the sequence $\lambda_n$ is separated from $\frac{1}{e}$ has been covered by Urbański and Zdunik in \cite{ZU}. In particular they have proved the following theorem.

\begin{theorem}\label{zdunik}
Let $\forall_n M > \lambda_n > \bar{\lambda} > \frac{1}{e}$ for some fixed constants $M , \bar{\lambda}$. Then $J(\lambda_n) = \mathbb{C}$. 
\end{theorem}

The following propositions can both be seen as corollaries of the proof of the above theorem.

\begin{proposition}\label{cora}
Let $\forall_n M > \lambda_n > \bar{\lambda} > 0$ for some constants $\bar{\lambda}, M > 0$ and let $\forall_{x \in \mathbb{R}} \lim\limits_{n \rightarrow \infty} F_n (x) = \infty$. If there is no open set $U$ such that there exists a subsequence of $F_{n}$ that converges on $U$ almost uniformly to a constant $a \in \mathbb{R}$, then $J(\lambda_n) = \mathbb{C}$.
\end{proposition}

\begin{proposition}\label{corjed}
Let $M > \lambda_n \geqslant \frac{1}{e}$ for some constant $M > 0$ and $\forall_{x \in \mathbb{R}} \lim\limits_{n \rightarrow \infty} F_n (x) = \infty$. If there is no open set $U$ such that $F_{n}$ converges on $U$ almost uniformly to $1$, then $J(\lambda_n) = \mathbb{C}$.
\end{proposition}

It is perhaps important to note that both of the propositions follow from the original authors' work in \cite{ZU}, but are not stated by them explicitly. Theorem \ref{zdunik} is Theroem 7 from \cite{ZU}, and the proof of both propositions \ref{cora}, \ref{corjed} is nearly identical to that of Theorem 7 from \cite{ZU}. The only meaningful difference is that in the propositions we explicitly assume convergence to infinity on the real line, and we exclude the possibility of a Fatou component converging to a constant on the real line. These two properties are proved to be true for the sequences considered by Urbański and Zdunik as part of the proof of theorem \ref{zdunik}, but will not necessarily be true for more general sequences we shall consider in this document. This is why we reformulate the theorem in the form of these propositions, to have a useful criterion for when the Julia set is the whole plane. The proof of Proposition \ref{cora} (and, by extension, of Theorem \ref{zdunik} and Proposition \ref{corjed}) can be found in the appendix. It is almost word for word the exact proof provided by the authors of \cite{ZU}.

In other words, the propositions simply state that if we can somehow replace Lemma $14$ from \cite{ZU} with some other statement, which excludes constant limits on the real line, then one can repeat the proof in \cite{ZU} for any given sequence, without the strong assumption of it being bounded from below by a constant separated from $\frac{1}{e}$. Finding such statements is one of the goals of this document.

\section{The Julia set for general $(\lambda_n)$}
We shall discuss various ways in which one can weaken the assumptions of Theorem \ref{zdunik}, and consider more general sequences of $(\lambda_n)$. The following subsection concerns sequences not converging to $\frac{1}{e}$. In a sense, the next two theorems, along with Theorem \ref{zdunik}, deal with cases similar to the autonomous iteration for $\lambda > \frac{1}{e}$.

\subsection{Sequences of $(\lambda_n)$ that do not converge to $\frac{1}{e}$}

\begin{theorem}
Let $\forall_n \lambda_n \in (\bar{\lambda}, M) \cup \{\frac{1}{e} \}$, where $\bar{\lambda} > \frac{1}{e}$ and $M < \infty$, then the equality $J(\lambda_n) = \mathbb{C}$ holds if and only if $\lambda_n > \bar{\lambda}$ for infinitely many $n$.
\end{theorem}

\begin{proof}
First we assume that $\lambda_n > \bar{\lambda}$ only for finitely many $n$. In this case for sufficiently large $N$ we have 
$$
    \forall_{n > N} \lambda_n = \frac{1}{e}.
$$
which means the Fatou set is a preimage of the Fatou set of $e^{z-1}$ under a certain composition. Since the Fatou set of $e^{z-1}$ is non-empty, this concludes the proof of the easier implication. \\

Now assume $\lambda_n > \bar{\lambda}$ for infinitely many $n$. We shall show that 
\begin{equation}\label{geneq1}
    \forall_{x \in \mathbb{R}} \lim\limits_{n \rightarrow \infty} F_n (x) = \infty.
\end{equation}
and that there is no such open set $U$ on which $F_n \rightarrow 1$ almost uniformly. Then by Proposition \ref{corjed} we will be done.

We begin by checking the first condition. For $x > 1$ iterates of $e^{x-1}$ alone converge to infinity. If $x \leqslant 1$, then iterates of $e^{x-1}$ converge to $1$, which means that for some $n_0$ we will have $F_{n}(x) > \frac{e}{\bar{\lambda}}$ for all $n > n_0$. Now for $x > \frac{e}{\bar{\lambda}}$ we have $\bar{\lambda} e^x > 1$. Since we assume that $\lambda_n > \bar{\lambda}$ for infinitely many $n$, this means that for some index $n_1 > n_0$ we have $F_{n_1}(x) > 1$, and from this point the iterates converge to infinity. Thus indeed \eqref{geneq1} holds.

For a fixed $\varepsilon, v_1, v_2$ let us denote the sets 
$$
S_{\varepsilon} = \{z \in \mathbb{C} : |\Im(z)| < \varepsilon \}
$$

$$
    V_{\varepsilon} = \{z \in \mathbb{C} : |\Im(z)| < \varepsilon, v_1 < \Re(z) < v_2 \}.
$$
Now we shall set constants as follows:
\begin{enumerate}
    \item \label{it1} Let $v_1 < 1$ and $\varepsilon$ be such that $\forall_{z \in S_{\varepsilon}, \Re(z) > v_1} \Re(\bar{\lambda} e^z) > v_2 > 1$ for some constant $v_2 > 1$ (which we also fix at this moment)
    \item \label{it2} If necessary decrease $\varepsilon$ further so that $\forall_{z \in S_{\varepsilon}, \Re(z) > v_1} \Re(e^{z-1}) > v_1$
    \item \label{it3} Finally if necessary decrease $\varepsilon$ so that $\forall_{z \in S_{\varepsilon}, \Re(z) > v_2} \Re(e^{z-1}) > \Re(z) + \delta$ for some constant $\delta > 0$
\end{enumerate}

Let us assume that there is an open set $U$ on which $F_n$ converge uniformly to $1$, then after a finite number of iterations $N_1$ we should have
$$
    F_{N_1} (U) \subset V_{\varepsilon}.
$$
and by the definition of $V_{\varepsilon}$ this yields
$$
    \forall_{n > N_1} \forall_{z \in U} \Re(F_{n} (z)) > v_1 .
$$

Also by our assumptions we know that there is an index $N_2 > N_1$ such that $\lambda_{N_2} > \bar{\lambda}$, which by (\ref{it1}) gives us
$$
    \forall_{z \in U} \Re(F_{N_2} (z)) > v_2 .
$$
and by (\ref{it3})
$$
    \forall_{z \in U} \forall_{n > N_2} \Re(F_{n} (z)) > \Re(F_{n-1} (z)) + \delta.
$$
Thus the real part of the iterations increase to infinity, and any convergence to a real constant would be a contradiction.
\end{proof}

The following theorem is in a similar vein.

\begin{theorem}\label{randSeq}
Let $0 < \delta < \frac{1}{e}$. Consider a sequence of $\lambda_n$ chosen randomly with uniform distribution from the interval $(\frac{1}{e} - \delta , \frac{1}{e} + \delta)$, in the sense that the sequence is chosen from the product space of $(\frac{1}{e} - \delta , \frac{1}{e} + \delta)^{\mathbb{N}}$ with the usual product measure. Then for almost every sequence $(\lambda_n)$ we have $J(\lambda_n) = \mathbb{C}$.
\end{theorem}

\begin{proof}
The proof will be similar to that of the previous theorem, this time we would like to apply Proposition \ref{cora}. We shall show that the assumptions of the proposition hold for almost any sequence $(\lambda_n)$.

Let us fix a sequence $(\lambda_n)$. The function $(\frac{1}{e} - \delta)e^z$ has $2$ fixed points on the real line, which we shall denote as p and q. Similarly, for all $\lambda_n \in [\frac{1}{e} - \delta, \frac{1}{e})$ the function $\lambda_n e^z$ has $2$ real fixed points, which lie in the interval $[p,q]$. Consider the sets
$$
S_{\varepsilon} = \{z \in \mathbb{C} : |\Im(z)| < \varepsilon \}.
$$

$$
    V_{\alpha} = \{z \in \mathbb{C} : p - \alpha < \Re(z) < q + \alpha \}.
$$
and let us pick $\alpha, \varepsilon$ such that 

\begin{equation} \label{eqrnd}
    \forall_{z \in S_{\varepsilon} \setminus V_{\alpha}} \Re (f_{\lambda_n} (z)) > \Re(z) + \beta.
\end{equation}

for a constant $\beta$ independent of the choice of $\lambda_n \in (\frac{1}{e} - \delta, \frac{1}{e} + \delta)$, and also

\begin{equation} \label{eqstay}
    \forall_{z \in S_{\varepsilon} \cap V_{\alpha}} \forall_{\lambda_n \in (\frac{1}{e} - \delta, \frac{1}{e} + \delta)} \Re(f_{\lambda_n} (z)) > p - \alpha .
\end{equation}

Finally we may assume (by shrinking $\varepsilon$ if necessary) that we also have 

\begin{equation} \label{eqpush}
    \forall_{z \in S_{\varepsilon} \cap V_{\alpha}} \forall_{\lambda_n \in (\frac{1}{e} + \frac{\delta}{2}, \frac{1}{e} + \delta)} \Re(f_{\lambda_n} (z)) > \Re(z) + \beta.
\end{equation}

Assume now that indeed there is a constant on the real line which is the limit of a subsequence of $F_n$, on some open set $U$. We exclude the possibility of non-real limits for all subsequences (by proof of Proposition \ref{cora}, specifically this is statement \ref{wykgran}), which implies
$$
    \exists_N \forall_{n > N} F_n (U) \subset S_{\varepsilon}.
$$
which combined with \eqref{eqrnd} and \eqref{eqstay} yields
$$
    \exists_{N_1 > N} \forall_{n > N_1} F_{n} (U) \subset S_{\varepsilon} \cap \{\Re(z) > p - \alpha \}.
$$
Assume now that for a given sequence $\{\lambda_n \}$ there is a string of values $\lambda_n \in (\frac{1}{e} + \frac{\delta}{2}, \frac{1}{e} + \delta)$ of length at least $\frac{q - p + 2 \alpha}{\beta}$, that begins at the index $n_2 > N_1$. Then \eqref{eqpush} would imply that $$\forall_{z} \forall_{n > n_2 + \frac{q - p + 2 \alpha}{\beta}} \Re(F_n (z)) > q + \alpha .$$ This combined with \eqref{eqrnd} means that for such a sequence there can be no convergence to  constants on the real line, as the real part is increased by $\beta$ with all iterates.

Finally note that the set of sequences $\{\lambda_n \}$ containing a fixed string of $$\lambda_n \in (\frac{1}{e} + \frac{\delta}{2}, \frac{1}{e} + \delta)$$ of length at least $\frac{q - p + 2 \alpha}{\beta}$, at arbitrarily large indices, is a set of full measure, which concludes the proof. 

\end{proof}

It is worth noting that the uniform distribution does not play any key role in the reasoning above. In fact the proof of the previous theorem yields the following corollary:

\begin{cor}
Let $0 < \delta < \frac{1}{e}$. Let $\mu$ be a Borel probability measure on $(\frac{1}{e} - \delta, \frac{1}{e} + \delta)$ such that $\mu ((\frac{1}{e},\frac{1}{e} + \delta)) > 0$. Then for almost every (with respect to the product measure of $\mu$ on $(\frac{1}{e} - \delta, \frac{1}{e} + \delta)^{\mathbb{N}}$) sequence $\{\lambda_n \} \subset (\frac{1}{e} - \delta, \frac{1}{e} + \delta)^{\mathbb{N}}$ we have $J(\lambda_n) = \mathbb{C}$.
\end{cor}

\subsection{Sequences $(\lambda_n)$ converging to $\frac{1}{e}$}
Let us move on to the case of sequences of real $\lambda_n$ which converge to $\frac{1}{e}$ from above. From now on we shall use the notation $F_{n}^{n_1} = f_{\lambda_{n_1 + n}} \circ ... \circ f_{\lambda_{n_1 + 1}}$ for a composition of $n$ functions $f_{\lambda_n}$ starting at the index $n_1$. Considering such compositions is convenient for us, since the tail of the sequence $\{\lambda_n \}$ is what determines the Julia set. Thus some conditions for non-autonomous sequences have to be formulated in terms independent of the first finitely many $\lambda_n$. We start with the following lemma.

\begin{lemma}\label{lem}
If for a sequence $\{ \lambda_n \}$ with $\lambda_n > 0$ we have $\exists_{n_1} \forall_n F_n^{n_1}(0) < 1$ then $J(\lambda_n) \neq \mathbb{C}$
\end{lemma}

\begin{proof}
Consider an arbitrary point $z$ such that $\Re(z) < 0$. We have
$$
    \Re(f_{\lambda_n} (z)) = \cos(\Im(z)) \lambda_n e^{\Re(z)} \leqslant \lambda_n e^{\Re(z)} =  f_{\lambda_n} (\Re(z))
$$
which yields
$$
    \Re(F_n^{n_1} (z)) \leqslant F_n^{n_1} (\Re(z)) < F_n^{n_1} (0) < 1
$$
So after index $n_1$ all iterations of the entire half-plane $P_{-} = \{z: \Re(z) < 0 \}$ omit more than three points, and thus form a normal family by Montel's theorem. Then the appropriate preimage (by the first $n_1$ compositions) of any subset of the half-plane lies in the Fatou set.
\end{proof}

\begin{proposition}

We can choose a constant $C$ such that for $\lambda_n \leqslant \frac{1}{e} + \frac{C}{n^2}$ we have $J(\lambda_n) \neq \mathbb{C}$. 
\end{proposition}

\begin{proof}
Let us first check by induction that for $\lambda_n = \frac{1}{e} e^{\frac{1}{n-1}}(1 - \frac{1}{n})$ we get
$$
    F_n(0) = 1 - \frac{1}{n}.
$$
For $n = 1$ both sides of the equation are of course $0$. Now, assuming the equality for $n$, we get
$$
    F_{n+1} (0) = f_{\lambda_{n+1}} (F_n(0)) = \lambda_{n+1} e^{F_n(0)} = \frac{1}{e} e^{\frac{1}{n}} (1 - \frac{1}{n+1}) e^{1 - \frac{1}{n}} = 1 - \frac{1}{n+1}.
$$
This concludes the induction. Thus lemma \ref{lem} implies the Fatou set is non-empty for $\lambda_n \leqslant \frac{1}{e} e^{\frac{1}{n-1}}(1 - \frac{1}{n})$. It suffices now to pick an appropriate constant $C$.
Consider $C > 0$ such that

\begin{dmath}
    \frac{1}{e} + \frac{C}{n^2} < \frac{1}{e} + \frac{e^{-1}- \frac{1}{(n+1)e}}{2n^2} = \frac{1}{e}(1 + \frac{1}{2n^2} - \frac{1}{2n^2(n+1)}) = \frac{1}{e}(1 + \frac{1}{n} + \frac{1}{2n^2})(1 - \frac{1}{n+1}) < \frac{1}{e} e^{\frac{1}{n}}(1 - \frac{1}{n+1})
\end{dmath}
This concludes the proof.

\end{proof}

One can also choose $C$ such that for $\lambda_n = \frac{1}{e} + \frac{C}{n^2}$ all iterates of $0$ starting from any index converge to infinity, and thus Lemma \ref{lem} cannot be applied. Same is of course true for $\lambda_n = \frac{1}{e} + \frac{1}{n^p}$ for $p < 2$. This is roughly speaking a consequence of the fact that the distance between the $n$-th iterate of $0$ under $e^{z-1}$ and the fixed point at $1$ is of rate $\frac{1}{n}$.

One might be tempted to assume that the unboundedness of the trajectory of $0$ dictates whether the Julia set is the whole plane. The following theorem gives a counterexample, that is the Fatou set can be non-empty even for sequences where all iterates of $0$ beginning from any index diverge to infinity.

\begin{theorem}\label{przyk}
There exists a sequence $\{\lambda_n \}$ such that $ \forall_{n_1} F^{n_1}_n (0) = \infty$, but $J(\lambda_n) \neq \mathbb{C}$.
\end{theorem}

We shall require the following lemma.

\begin{lemma}
Let $f(z) = e^{z-1}$, then for every open set $V$ satisfying $\overline{V} \subset S:= \{z \in \mathbb{C}: 0 < \Re (z) < 1, 0 < \Im(z) < \frac{1}{2}  \}$ we have $\inf\limits_{z \in V, n \in \mathbb{N}}   \arg (f^n(z) - f^n (0)) > 0$, where $arg$ is the branch of the argument taking values from $[0, 2 \pi ]$.
\end{lemma}

In other words, the lemma says that we can fix an angle $\alpha$, such that the angle between the real line and a line segment $[f^n(0), f^n(z)]$ is at least $\alpha$, for any iterate $n$ and any point $z$ from $V$. 

\begin{proof}
In order to arrive at the proof we shall show that there exists a neighbourhood of $1$, such that for any two points $z, w$ in it, satisfying $w \in \mathbb{R}$, we have $\arg (e^{z-1} - e^{w-1}) > \arg (z - w)$. Indeed this is enough, since all iterates of both $0$ and the set $V$ will eventually land in the neighbourhood of $1$ under iteration of $e^{z-1}$, and never leave that neighbourhood. This means that after a certain number of iterations, the expression $\arg (f^n(z) - f^n(0))$ will be increasing with respect to $n$ for all $z$, thus indeed the infimum cannot be $0$. \\

Take a point $a \in \mathbb{R} \cap S$ and $z \in S$, let $x = \Re (z), y = \Im (z)$. Note that if $\Re(z-a) < 0$ then also $\Re(f^n (z) - f^n (a)) < 0$ for all iterations, and the argument of this difference is at least $\frac{\pi}{2}$. Assume then that $\Re(z-a) > 0$ and $\Re(e^{z-1} - e^{a-1}) > 0$. We want to have (for $a,z$ sufficiently close to $1$):

$$
\frac{\Im(z)}{\Re(z - a)} < \frac{\Im (e^{z-1})}{\Re (e^{z-1} - e^{a-1})}
$$
which can be rewritten as
$$
\frac{y}{x -a} < \frac{e^{x-1} \sin (y)}{e^{x-1} \cos(y) - e^{a - 1}}.
$$

Denote $x' = x - a$, we need to prove

$$
\frac{y}{x'} < \frac{\sin(y)}{\cos(y) - e^{-x'}}
$$
which is equivalent to
$$
y(\cos(y) - e^{-x'}) < \sin (y) x'.
$$

Passing to Taylor series will now be enough to see that the last inequality is true:

$$
y\left((1 - \frac{y^2}{2} + \frac{y^4}{4!} - ...) - (1 - x' + \frac{{x'}^2}{2} - ....)\right) <x' (y - \frac{y^3}{3!} + ...)
$$
which can be rewritten as
$$
x' - \left(\frac{y^2}{2} - ...\right) - \left(\frac{{x'}^2}{2} - ...\right) < x' - \frac{x' y^2}{3!} + ....
$$

Last inequality is true for $x',y$ sufficiently close to $0$. Since $y = \Im (z)$ and $x' = \Re (z) - a$ both converge to $0$ under iteration, this concludes the proof.

\end{proof}

\begin{proof}[Proof of theorem \ref{przyk}]
Let us pick an open set $V$ such that the assumptions of the previous lemma are satisfied. We shall build the sequence $\{\lambda_n \}$ in a way that $V \subset F (\lambda_n)$. The sequence we choose will contain long strings of $\frac{1}{e}$, that is $\lambda_n = \frac{1}{e}$ for $n \in [M_{k} + 1, M_{k+1} - 1]$, where the numbers $M_k$ shall be set during the construction. The values of $\lambda_{M_k}$ we will pick larger than $\frac{1}{e}$, to ensure that the iterates on the real line escape to infinity. Let us now present the construction. \\

Let again $f(z) = e^{z-1}$, then for any $\frac{\pi}{2} > \alpha_1 > 0$ there exists $\varepsilon$ such that $$\{z \in \mathbb{C} : \Re (z-1) > 0,  \Im(z) < \varepsilon, \arg(z-1) > \alpha_1 \} \subset F (f).$$ Indeed, it is well known that the Julia set of $f$ is tangent to the real line at $1$, thus for any angle $\alpha_1$ we can choose an appropriate $\varepsilon$. Now for our set $V$ let $\alpha_1$ be a number satisfying $$\inf\limits_{z \in V, n \in \mathbb{N}}   \arg (f^n(z) - f^n (0)) > \alpha_1 > 0,$$ the existence of it is given by our previous lemma. We pick $\varepsilon$ for $\alpha_1$ as we discussed above. Finally we can set $M_1$ large enough so that $$f^{M_1 - 1} (V) \subset \{z: \Im(z) < \varepsilon \}.$$ Now we choose $\lambda_{M_1}$ so that $F^{M_1}(0) = 1$. This means that $\lambda_{M_1} > \frac{1}{e}$, since otherwise we would have $F^{M_1}(0) < 1$. Moreover because of the choices of $\alpha_1, \varepsilon$, we know that $F^{M_k}(V)$ omits the autonomous Julia set $J(e^{z-1})$, since $F^{M_k}(V)$ lies outside of the angle $\alpha_1$. Thus there is a natural number $k_1$ such that we have $$f^{k-1}(F^{M_1})(V) = F^{M_1 + k_1} (V) \subset \{\Re (z) < 1 \}.$$ This is a consequence of the fact that in general points in the Fatou set of $e^{z-1}$ are almost uniformly moved to $\{\Re (z) < 1 \}$. \\
Now it is enough to repeat this construction for the set $F^{M_1 + k_1} (V)$, by picking sufficiently large $M_2$ for the new angle $\alpha_2$ (which maybe depends on the position of $F^{M_1 + k_1} (V)$), and an appropriate $\lambda_{M_2}$ so that $F_{M_{n+2}}^{M_{n+1}+1}(0) = 1$ and $F^{M_1 + M_2} (V) \subset F(e^{z-1})$. This way we build our sequence inductively by picking all $M_n$ and $\lambda_{M_n}$. \\
It remains to be shown that this sequence satisfies our requirements. The set $V$ is in the Fatou set by Montel's theorem, since by the construction all its iterates omit the entire real line. \\
The iterates of $0$ starting from any index converge to infinity. Indeed, if we start from say $k \in [M_n, M_{n+1})$, then $$F_{M_{n+2}}^{k}(0) > 1,$$ because $$F_{M_{n+2}}^{M_{n+1}+1}(0) = 1$$ and $$F^k_{M_{n+1}}(0) > 0.$$ Since the iterates of $f^n(z)$ itself for any point $z > 1$ converge to infinity, then of course same is true for the non-autonomous sequence we constructed (where we iterate either $f$ or $f$ multiplied by constant bigger than $1$). This concludes the proof.

\end{proof}

Let us point out that the Fatou set produced in the above construction differs significantly from the autonomous Fatou set for $e^{z-1}$. This can be seen by the fact that the entire real line along with all its preimages has to be in the Julia set. We now turn to another construction, this time we would like the Fatou set to be empty, but the sequence $\lambda_n$ to be as cloes to $\frac{1}{e}$ as possible.
 
\begin{proposition}
We construct a sequence $\lambda_n \searrow \frac{1}{e}$ such that the Julia set is the whole plane. The sequence we pick satisfies $\limsup\limits_{n \rightarrow \infty} (\lambda_n - e^{-1})n^{\frac{1}{2}} = const$. 
\end{proposition}
\begin{proof}
Let us denote 
    \begin{equation}
    P = \{z \in \mathbb{C}: \frac{1}{2} <  \Re (z) < \frac{3}{2} , 0 <  \Im (z) < \frac{1}{2} \}
    \end{equation}
    \begin{equation}
    S_n = \{z: \in \mathbb{C} : \frac{1}{2} < \Re (z) < \frac{3}{2}, 0 < \Im (z) < \varepsilon_n \} \subset P.
    \end{equation}
and consider $\lambda_{k} = \frac{1}{e} + \frac{1}{n}$ for $k \in [\sum\limits_{i = 1}^{n-1} M_i, \sum\limits_{i = 1}^{n} M_i]$, where $M_n$ will be set in the construction. That is, we shall arrive at the desired sequence by repeating each $\frac{1}{e} + \frac{1}{n}$ an appropriate $M_n$ number of times.

Let us for now fix $n$ and consider $f_{\lambda} (z) = (\frac{1}{e} + \frac{1}{n})e^z$. Let us choose $\varepsilon_n$ such that $f_{\lambda} (z) = \lambda e^z$ increases the real part by a fixed amount $\beta_n$ on $S_n$, that is 
$$
    \exists_{\beta_n > 0} \forall_{z \in S_n} \Re (f_{\lambda}(z)) > \Re (z) + \beta_n  .
$$
 It now follows that since the width of the strip $S_n$ is finite we get
$$
    \exists_{K_n} \forall_{z \in P} \exists_{k \in \mathbb{N}, 0 < k \leqslant K_n, F_{k}(z) \in \mathbb{C} \setminus S_n}.
$$
In particular, if we assume all the iterates of $f_{\lambda}$ lie in $P$, then this means that at least once every $K_n$ iterates a point from $P$ lands in $P \setminus S_{n}$. 

Now let $\varrho |dz|$ be the hyperbolic metric on $S = \{0 < \Im (z) < \pi \}$ and let the hyperbolic derivative of $f_{\lambda}$ be denoted by $|f_{\lambda}'(z)|_{\varrho}$. Then we have
\begin{equation}\label{alfa}
    \exists_{\alpha_n} \forall_{z: f_{\lambda} (z) \in P \setminus S_n} |f_{\lambda} (z) '|_{\varrho} > \alpha_n > 1
\end{equation}
and $|f_{\lambda}'(z)|_{\varrho} \geqslant 1$ for other $f_{\lambda}(z) \in P$, from the Schwarz lemma applied to $f_{\lambda}^{-1}: \mathbb{H}^+ \rightarrow S$ (where $\mathbb{H}^+$ is the upper half-plane). Combined with the previous observation, if we assume that all iterates lie in $P$, then this implies
\begin{equation}\label{alfak}
    \forall_{M > K_n} |f^M_{\lambda}  (z)'|_{\varrho} > \alpha_n^{\frac{M}{K_n}}.
\end{equation}
Recall that for any holomorphic function $f$ we have 

$$
    |f'(z)| = |f'(z)|_{\varrho} \frac{\varrho (z)}{\varrho (f(z))} 
$$

which yields the following bound

\begin{equation}\label{hyper}
    \exists_{C_n > 0} \forall_{z: f(z) \in P \setminus S_n} |f'(z)| > |f'(z)|_{\varrho} C_n
\end{equation}
since the metric $\varrho$ is bounded from above on $P \setminus S_n$ and bounded from below everywhere in $P$ (so we just take $C_n = \inf\limits_{f(z) \in P \setminus S_n} \frac{\rho(z)}{\rho(f(z))}$). This holds for all holomorphic functions, in particular any compositions of the exponential functions that are of interest to us.

Consider a sequence $\delta_n$ with $\lim\limits_{n \rightarrow \infty} \delta_n = 0$, and let us pick $M_n$ large enough so that 
\begin{equation}\label{eq:1}
    \delta_n > 4 \frac{\text{diam}(P)}{C_n \prod\limits_{k = 1}^{n} \alpha_k^{\frac{M_k}{K_k}}},
\end{equation}
and further enlarge $M_n$ by an additional $K_n$ (in particular this means now that that $M_n > K_n$ holds for all $n$).
We claim that the sequence in which each $\frac{1}{e} + \frac{1}{n}$ is repeated $M_n$ times is our desired sequence for which the Julia set is the whole plane. \par

Indeed, assume there exists an open set $V \subset F(\lambda_n)$, we can assume 
$$
    \forall_{n \geqslant 0} F_n(V) \subset P
$$
since we know that $F_n$ converge on $V$ to $1$ from Proposition \ref{corjed}. Now let $z \in V$ and let $D(z,r) \subset V$ be an open disk around $z$ with radius $r$. We denote $$V_n = F_{\sum\limits_{k = 1}^{n} M_k} (V).$$ Now let us note that the Koebe one quarter theorem combined with 
$$
    \forall_n V_n \subset P 
$$
and \eqref{eq:1} implies that
$$\forall_n r < \delta_n.$$

Indeed, let us assume $r \geqslant \delta_n$, by Koebe one quarter theorem we have

$$
    F_{\sum\limits_{k = 1}^{n} M_k}(D(z,r)) \supset D\left(F_{\sum\limits_{k = 1}^{n} M_k}(z), \frac{1}{4} r |F_{\sum\limits_{k = 1}^{n} M_k}'(z)|\right)
$$

and now \eqref{alfak}, \eqref{hyper} and \eqref{eq:1} yield

$$
    \frac{1}{4} r |F_{\sum\limits_{k = 1}^{n} M_k}'(z)| \geqslant \frac{1}{4} \left(4 \frac{\text{diam}(P)}{C_n \prod_{k=1}^{n} \alpha_k^{\frac{M_k}{K_k}}}\right) |F_{\sum\limits_{k = 1}^{n} M_k}'(z)|_{\varrho} C_n > \text{diam}(P).
$$
The second inequality is just \eqref{alfak}. The first inequality is the comparison $r \geqslant \delta_n$ combined with \eqref{hyper} and \eqref{eq:1}. Applying \eqref{hyper} is valid here, since without loss of generality we may assume the iterates $F_{\sum\limits_{k = 1}^{n} M_k}(z)$ leave $S_n$, otherwise we can consider a number $N_n (z) \in [M_n - K_n, M_n]$ such that $F_{N_n(z) + \sum\limits_{k = 1}^{n-1} M_k}(z)$ leaves $S_n$. Indeed, by construction of $M_k$ such a choice of $N_k(z)$ is always possible (for all $z$), and then we have the following inequality

\begin{align*}
    \frac{1}{4} r |F_{N_n + \sum\limits_{k = 1}^{n-1} M_k}'(z)| \geqslant \frac{1}{4} \left(4 \frac{\text{diam}(P)}{C_n \alpha_n^{\frac{N_n}{K_k}} \prod_{k=1}^{n-1} \alpha_k^{\frac{M_k}{K_k}}}\right) |F_{N_n(z) + \sum\limits_{k = 1}^{n-1} M_k}'(z)|_{\varrho} C_n \\ > \text{diam}(P).
\end{align*}

This yields the desired contradiction, since we have shown that if $r \geqslant \delta_n$, then $V_n$ contains a ball too large to fit in $P$.

Since $n$ was arbitrary, and we have picked the numbers $\delta_n$ to converge to zero, there can be no open disk around $z$ in $V$, and thus $V$ is not an open set which yields the contradiction.

Now let us calculate exactly how quickly does the example sequence $\lambda_n$ converge to $0$. To make any estimate we first need to set appropriate $\varepsilon_n , M_n , K_n , C_n, \alpha_n , \delta_n$. Note that the hyperbolic metric $\varrho$ is given by the density $\varrho (z) = \frac{1}{\sin(\Im(z))}$. Indeed, since the exponential function sends the strip $\{ 0 <  \Im(z) < \pi \}$ onto the upper halfplane, which has the hyperbolic metric $\varrho_H (z) = \frac{1}{\Im(z)}$, this yields: \\
$$\varrho(z) = \varrho_H(\exp(z)) |\exp ' (z)| = \frac{|\exp ' (z)|}{\Im (\exp(z))} = \frac{1}{\sin(\Im (z))}$$. \\

Now let us take $\varepsilon_n = \frac{1}{\sqrt n}$. In that case for $z \in S_n$ we have \\
\begin{dmath*}
    \Re(e^{z -1} (1 + \frac{1}{n})) = e^{\Re(z)-1} \cos(\Im(z)) (1 + \frac{1}{n}) > \Re (z) (1 - \frac{\Im(z)^2}{2})(1 + \frac{1}{n}) > \Re(z)(1 - \frac{1}{2n})(1+\frac{1}{n}) > \Re(z)(1 + \frac{1}{3n}) > \Re(z) + \frac{1}{6n}.
\end{dmath*} 
Since the length of the strip $S_n$ is finite, this means that for some constant $C$ after $Cn$ steps the iterates of $z$ leave $S_n$. In fact since $S_n$ has length $1$, we can pick $K_n = 6 n$. \\ We move on to the choice of $\alpha_n$. Let $\varrho_{H}$ be the hyperbolic metric on the upper halfplane and $\varrho_S$ be the hyperbolic metric on the strip from $0$ to $\pi$ (note that $\varrho = \varrho_S$, in this paragraph we shall use the notation $\varrho_S$ to emphasize that we are dealing with hyperbolic metrics on $2$ different domains $H$ and $S$). Since the exponential function is an isometry between these two metrics, we have $|f_{\lambda}(z)'| \frac{\varrho_{H}(f(z))}{\varrho_S (z)} = 1$. This lets us calculate the hyperbolic derivative of $f_{\lambda}$ in the strip, since we get:
\begin{multline*}
|f_{\lambda}(z)'|_{\varrho} = |f_{\lambda}(z)'| \frac{\varrho_S (f(z))}{\varrho_S (z)} = |f_{\lambda}(z)'| \frac{\varrho_H (f(z))}{\varrho_S (z)} \frac{\varrho_S (f(z))}{\varrho_H (f(z))} = \frac{\varrho_S (f(z))}{\varrho_H (f(z))} \\ = \frac{\Im(f(z))}{\sin(\Im(f(z)))}.
\end{multline*}

Now we can write (for sufficiently large $n$)
\begin{equation*}
\begin{split}
\inf\limits_{f(z) \in P \setminus S_n} |e^{z-1} (1+\frac{1}{n})| |\frac{\varrho_S (f(z))}{\varrho_S (z)}| &= \inf\limits_{f(z) \in P \setminus S_n} |\frac{\Im(f(z))}{\sin(\Im (f(z)))}| \\  &> \frac{\varepsilon_n}{\varepsilon_n - \frac{\varepsilon_n^3}{7}} > \frac{1}{1 - \frac{1}{7n}} > 1+\frac{1}{7n}.
\end{split}
\end{equation*}
This means $\alpha_n = 1 + \frac{1}{7n}$ is a valid choice for the estimate in \eqref{alfa}. \\

Lastly let us note that since the hyperbolic metric is bounded by a constant times distance from the boundary of the set, we can write $C_n > \frac{C'}{\sqrt n}$ for some constant $C'$ independent of $n$. \\

If we now pick $M_n = 7 K_n + K_n=  48n$, we get
$$
    C^{''}\frac{\sqrt n}{\prod_{k=1}^{n} (1 + \frac{1}{7k})^7} > 4 \frac{\text{diam} P}{C_n \prod_{k=1}^{n} \alpha_k^{\frac{M_k}{K_k}}} 
$$
for some constant $C^{''}$ independent of $n$.
Note that the left hand side converges to $0$. Indeed, we have 
$$
\prod_{k=1}^{n} (1 + \frac{1}{7k})^7 > \prod_{k=1}^{n} (1 + \frac{1}{k}) = n + 1.
$$
This means we can choose $\delta_n$ such that $\delta_n > C^{''}\frac{\sqrt n}{\prod_{k=1}^{n} (1 + \frac{1}{7k})^7}$ while keeping $\lim\limits_{n \rightarrow \infty} \delta_n =  0$, and thus our choice of all the numbers $\varepsilon_n , M_n , K_n , C_n, \alpha_n , \delta_n$ is valid. 
  Thus we have found an explicit example sequence, in which $\lambda = \frac{1}{n}$ is each repeated $\lceil 48n \rceil$ times. Since $48 \sum_{k=1}^n k \approx n^2$, the rate of convergence is $\frac{1}{n^{\frac{1}{2}}}$.
\end{proof}

Actually, the above construction works without any modification if we take any $\lambda \geqslant \frac{1}{n}$ repeated $48n$ times, equality isn't necessary. Indeed, all that is needed is some bound for the expansion of the hyperbolic metric. Thus the proof yields the following corollary:

\begin{cor}
If $\lambda_n = \frac{1}{e} + \frac{1}{n^p}$ for $p < \frac{1}{2}$, then $J (\lambda_n) = \mathbb{C}$.
\end{cor}

We do not conclude whether the Julia set for $\lambda_n = \frac{1}{e} + \frac{1}{n^p}$ with $\frac{1}{2} < p < 2$ is the whole plane, but if not, then the behaviour on the Fatou set is vastly different from the autonomous case, or the case for $p > 2$. For instance, we know that the real line along with its preimages lies in the Julia set. Moreover any components of the Fatou set would have to converge to $1$, the fixed point for the autonomous system. In the autonomous case the iterations on the left half-plane keep getting closer to the real line, never leaving a certain cone. The following theorem points to a different behaviour for the non-autonomous case which we are considering.

\begin{theorem}\label{stozek}
Let $S_{\theta} = \{z \in \mathbb{C}: \arg(z) \in (\frac{\pi}{2} + \theta, \frac{3 \pi}{2} - \theta) \}$ where $\theta \in (0, \frac{\pi}{2})$ and let $f_{\lambda_n} = e^z - 1 + \frac{1}{n^p}$ where $p < 2$. Then for every $z \in S_{\theta}$ there exists $n$ such that $F_n (z) \in \mathbb{C} \setminus S_{\theta}$ 
\end{theorem}
The above theorem is formulated in terms of the function $e^z - 1$ and not $e^{z-1}$. These are conjugated, we change the variables to have the fixed point at $0$ instead of $1$.

\begin{lemma}\label{stozeklem}
Let $f(z) = e^{z-1}$. Then $$\forall_{\theta \in (0, \frac{\pi}{2})} \exists_{r, C_1, C_2} \forall_{z \in S_{\theta} \cap B(0,r)} : |f(z)| < |z|$$ and $$C_1 < \liminf\limits_{n \rightarrow \infty} (|f^n(z)|n) \leqslant \limsup\limits_{n \rightarrow \infty} (|f^n(z)|n) < C_2$$ 
\end{lemma}

\begin{proof}
Let us consider the function $f$ after a change of variables given by $\frac{1}{z}$. Then we have
$$
    g(z) = \frac{1}{f(\frac{1}{z})} = \frac{1}{e^{\frac{1}{z}-1}} = \frac{z-2}{1 + O(\frac{1}{z^2})} + \frac{4}{z + 2 + O(\frac{1}{z})}.
$$
The function $\frac{1}{z}$ is a bijection between $S_{\theta} \cap B(0,r)$ and $S'_{\theta} = S_{\theta} \cap (\mathbb{C} \setminus B(0,\frac{1}{r}))$. Let us take $z \in S'_{\theta}$, and let us assume that $Arg(z) = \frac{\pi}{2} + \theta$, since the proof for other $z$ follows from this case. Of course $z$ lies on the circle $B(0,|z|)$, let us denote the line tangent to $B(0,|z|)$ at point $z$ by $k$. The angle between the line segment $[z,z-2]$ and $k$, is also equal $\theta$. 
From the form of $g(z)$ we know that for sufficiently small $r$, the angle between the line segment $[z,g(z)]$ and $k$ is at least $\frac{\theta}{2}$. Indeed, we have

\begin{figure}[htp]
\centering
\includegraphics[width=10cm]{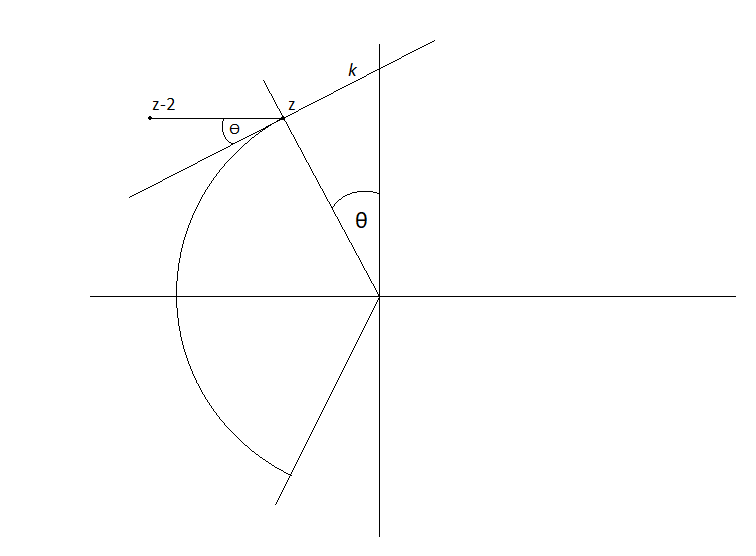}
\caption{}
\label{fig:drawing}
\end{figure}
$$
    \lim\limits_{z \rightarrow \infty} |g(z) - (z-2)| = 0.
$$
while
$$
     \text{dist}(z-2,k) = 2 \sin ( \theta ).
$$
This yields 
$$
    g(S'_{\theta}) \subset S'_{\theta} .
$$
and moreover, by taking a smaller $r$ if necessary, we can get
$$
    |z| + \sin (\frac{\theta}{2}) < |g(z)| < |z| + 3.
$$
Now since $g(S'_{\theta}) \subset S'_{\theta} $ we can also write the inequalities for all iterates
$$
    \forall_n |z| + 3n > | g^n (z)| > |z| + n \sin (\frac{\theta}{2}).
$$
which gives the desired result after returning to the initial coordinates. 

\end{proof}

\begin{proof}[Proof of theorem \ref{stozek}]
Let us pick the appropriate $r$ for $\theta$ so that the inequalities from Lemma \ref{stozeklem} are satisfied. We assume all the iterates of a point $z$ stay in $S_{\theta}$ and aim to arrive at a contradiction. Let us start by showing that if all iterates of $z$ stay in $S_{\theta}$, then for some $n$ large enough $F_n(z)$ must land in $B(0,r)$: 

Let us denote $S''_{\theta} = \overline{B(0,3)} \cap S_{\theta}$ and again let $f(z) = e^z - 1$. It is well known that the entire left halfplane is in the basin of attraction of $0$ for the function $f$, thus $f^n$ converge uniformly to $0$ on $S''_{\theta}$. This yields
$$
    \exists_{n_1} f^{n_1} (S''_{\theta}) \subset B(0, \frac{r}{2})
$$
which means that for a sufficiently large $n_2$ we have
$$
    f_{\lambda_{n_2 + n_1}} \circ f_{\lambda_{n_2 + n_1 - 1}} \circ ... \circ f_{\lambda_{n_2}} (S''_{\theta}) \subset B(0,r).
$$
since of course
$$
    \forall_{\varepsilon} \exists_{n_2} \forall_{n > n_2} \forall_{z \in S''_{\theta}} |f_{\lambda_n} (z) - f(z)| < \varepsilon.
$$
Note also that
$$
    \forall_n \forall_{z : \Re(z) < 0} |f_{\lambda_n} (z)| = |e^z - 1 + \frac{1}{n^p}| < |e^z| + 2 < 3 .
$$
which implies $\forall_n \forall_{z : \Re(z) < 0} f_{\lambda_n} (z) \in S''_{\theta}$ thus finally yielding $F_{n_1 + n_2} (z) \in B(0,r)$. \\

Knowing that the iterates land in $B(0,r)$ after $n_1 + n_2$ iterations will allow us to apply Lemma \ref{stozeklem}. We can now move on to showing that $F_n(z)$ has to leave $S_{\theta}$ at some point. We assume all iterates stay in $S_{\theta}$ and arrive at a contradiction by showing that the modulus of some iterates would have to be smaller than $0$. \\ Let us first note that we may assume the inequality $\limsup\limits_{n \rightarrow \infty} (|f^n(z)|n) < C_2$ from Lemma \ref{stozeklem} is true also for the compositions of $f_{\lambda_n}$ and not just for $f$. Indeed, this follows from the fact that $$(z \in S_{\theta}) \wedge (z + \frac{1}{n^p} \in S_{\theta}) \implies |z + \frac{1}{n^p}| < |z|,$$ i.e. if both $f(z)$ and $f(z) + \frac{1}{n^p}$ stay in $S_{\theta}$, then the latter has a smaller modulus. Thus assuming all iterates $F_n(z)$ stay in $S_{\theta}$ Lemma \ref{stozeklem} gives us 
$$
    \exists_{M > 0, N > 0} \forall_{n > N} \forall_{z \in S_{\theta}} |F_{n + n_1 + n_2}(z)| < \frac{M}{n}.
$$
We now have

\begin{dmath}
    \exists_{M> 0, C>0 } \forall_{m > M} \abs{F_{n_1 + n_2 + n + m}(z)} < \abs{F_{n_1 + n_2 + n}(z)} - C \sum\limits_{k = n_1 + n_2 + n}^{n_1 + n_2 + n + m} \frac{1}{k^p} < \abs{F_{n_1 + n_2 + n}(z)} - C_1 \frac{1}{n^{p-1}} < \frac{M}{n} - C_1 \frac{1}{n^{p-1}} < 0 .
\end{dmath}

The last inequality holds for sufficiently large $n$. This contradiction concludes the proof.

\end{proof}

\appendix
\section{Proof of Propositions \ref{cora} and \ref{corjed}}

For sake of completeness we provide the following proof of Proposition \ref{cora}, which is just a minor modification of the proof of Theorem 7 given in \cite{ZU} (this is Theorem \ref{zdunik} in this document). The proof of Proposition \ref{corjed} is analogous.

\begin{proof}[Proof of Proposition \ref{cora}]
We begin by noting the real line lies in the Julia set.
\begin{lemma}\label{lemre}
Let $\forall_n M > \lambda_n > \bar{\lambda} > 0$ for some constants $M, \bar{\lambda}$, and let $\forall_{x \in \mathbb{R}} \lim\limits_{n \rightarrow \infty} F_n(x) = \infty$. Then $\mathbb{R} \subset J(\lambda_n)$.
\end{lemma}
\begin{proof}
Let us assume a point $w \in \mathbb{R}$ is in the Fatou set. Then there exists an open set $V$ such that $w \in V$ and the family $(F_n |_{V})$ is normal. Since $F_n |_{\mathbb{R} \cap V} \rightarrow \infty$ as $n \rightarrow \infty$, we conclude $F_n$ converges to infinity uniformly on compact subsets of $V$. Now consider a ball $B(w,r) \subset V$, we have
$$(F_n)'|_{B(w,r)} \rightarrow \infty$$
uniformly as $n \rightarrow \infty$. Thus by Bloch's Theorem, the image $F_n(B(w,r))$ for sufficiently large $n$ contains a ball of radius $2 \pi $. This implies there exists a sequence of points $z_n \in B(w,r)$ such that 

$$\lim\limits_{n \rightarrow \infty}|\Re (F_n (z_n))| = \infty$$

and

$$\Im (F_n(z_n))  \in \pi + 2 \pi \mathbb{Z}.$$
Thus $F_{n+1} (z_n) \in (-\infty , 0)$, which contradicts the convergence $F_n|_{B(w,r)} \rightarrow \infty$, and concludes the proof.
\end{proof}

A straightforward consequence of this lemma is the following corollary:

\begin{cor}\label{rekopie}
If $V \subset \mathbb{C}$ is an open set and $V \cap J(\lambda_n) = \emptyset$, then $V \cap \mathbb{R} = \emptyset$. Furthermore,

$$(\bigcup\limits_{k \in \mathbb{Z}} \mathbb{R} + k \pi i) \cap \bigcup\limits_{n=0}^{\infty}F_n(V) = \emptyset $$
\end{cor}

As the original authors note, the following lemma is actually due to Misiurewicz, and can be found in \cite{MM}.

\begin{lemma}\label{pochodna}
For every $z \in \mathbb{C}$ and every $n \geqslant 1$ we have 

$$|(F_n)'(z)| > |\Im F_n(z) |$$
\end{lemma}

\begin{proof}
We have $f_{\lambda} (z) = \lambda e^x (\cos (y) + i \sin(y))$. Since $|\sin(y)| \leqslant |y|$ we have $|\Im (f_{\lambda}(z))| \leqslant \lambda e^x |y| = |f_{\lambda}(z)||\Im (z)|$. So,

$$\frac{|\Im (f_{\lambda} (z))|}{|\Im(z)|} \leqslant |f_{\lambda} (z)|.$$
Therefore we have

\begin{dmath*}
   \Im(F_n(z)) = (\prod\limits_{k = 2}^{n} \frac{|\Im(F_k(z))|}{|\Im (F_{k-1}(z))|}) \cdot \Im (F_1 (z)) \\ \leqslant  (\prod\limits_{k=2}^{n}|F_k (z)|)|\Im (F_{1}(z))|  \\ \leqslant \prod\limits_{k=1}^{n}|F_k (z)| \\ = |(F_n)'(z)|
\end{dmath*}

\end{proof}

\begin{lemma}
If $V \subset \mathbb{C}$ is an open connected set and $V \subset \overline{V} \subset \mathbb{C} \setminus J(\lambda_n)$, then there exists an integer $N \geqslant 0$ such that for all $n \geqslant N,$

$$ F_n (V) \subset S:= \{ z \in \mathbb{C} : |\Im(z)| < \pi \} .$$
\end{lemma}

\begin{proof}
By corollary \ref{rekopie}, for every $n \in \mathbb{N}$, either the set $F_n(V)$ is contained in $S$, or it is disjoint from $S$. If $F_n(V) \cap S = \emptyset$ for infinitely many integers $n$, then using lemma \ref{pochodna} and the Chain Rule we obtain

$$
    \limsup\limits_{n \rightarrow \infty} |(F_n)'|_{V} = \infty.
$$
This by Bloch's Theorem implies $F_n(V)$ contains a ball of radius $2 \pi$ for infinitely many $n$, which is a contradiction with corollary \ref{rekopie}. The contradiction concludes the proof.

\end{proof}

Write $S$ as

$$
S = S^{+} \cup S^{-} \cup \mathbb{R},
$$
where

$$
S^{+} := \{z \in \mathbb{C}: 0 < \Im (z) < \pi \}
$$

and

$$
S^{-} := \{z \in \mathbb{C}: - \pi < \Im(z) < 0  \}.
$$

For a given $f_{\lambda}$ denote by $g_{\lambda}$ the branch of the holomorphic inverse of $f_{\lambda}$ which maps $S^{+}$ to $S^{+}$. Let $\rho$ denote the hyperbolic metric on $S^{+}$.

\begin{lemma}
For every $\lambda \in [\bar{\lambda},M]$ and for all $z,w \in S^{+}$, we have that
\begin{equation}\label{eq1}
\rho(g_{\lambda}(z),g_{\lambda}(w)) \leqslant \rho (z,w).
\end{equation}
Also, for every compact subset $K \subset S^{+}$ there exists $\kappa \in (0,1)$ such that for any $\lambda \in [\bar{\lambda}, \infty)$ and for all $z,w \in K$, we have
\begin{equation}\label{eq2}
\rho(g_{\lambda}(z),g_{\lambda}(w)) \leqslant \kappa \rho (z,w).
\end{equation}
\end{lemma}

\begin{proof}
The inequality (\ref{eq1}) is an immediate consequence of Schwarz Lemma. Since the map $g_{\lambda} : S^{+} \rightarrow S^{+}$ is not bi-holomorphic, it also follows from Schwarz lemma that
\begin{equation}\label{eq3}
    \rho(g_{\lambda}(z),g_{\lambda}(w)) < \rho (z,w),
\end{equation}
whenever $z,w \in S^{+}$ and $z \neq w$, and in addition,

\begin{equation}\label{eq4}
   \limsup\limits_{\substack{z,w \rightarrow \xi \\ z \neq w}} \frac{\rho (g_{\lambda}(z),g_{\lambda}(w))}{\rho (z,w)} < 1
\end{equation}
for every $\xi \in S^{+}$. In order to prove (\ref{eq2}), fix $\lambda_2 \geqslant \lambda_1 \geqslant \bar{\lambda}$. Since $g_{\lambda_2}(z) = g_{\lambda_1} (z) - \log (\frac{\lambda_2}{\lambda_1})$ and $g_{\lambda_2}(w) = g_{\lambda_1} (w) - \log (\frac{\lambda_2}{\lambda_1})$, and since the metric $\rho$ is invariant under horizontal translation, we have
$$
\rho(g_{\lambda_2}(z), g_{\lambda_2}(w)) = \rho(g_{\lambda_1}(z), g_{\lambda_2}(w)).
$$
Thus it is enough to check (\ref{eq2}) for $f_{\bar{\lambda}}$. But this follows immediately from (\ref{eq3}), (\ref{eq4}) and the compactness of $K$. Indeed, denote by $|f'|_{\rho}$ the derivative with respect to the metric $\rho$, and consider the function $G: K \times K \rightarrow \mathbb{R}$ defined by:
\begin{equation*}
    G(z,w) =
    \begin{cases}
      \frac{\rho (g_{\bar{\lambda}}(z),g_{\bar{\lambda}}(w))}{\rho (z,w)} & \mbox{for } z \neq w \\
      |f'|_{\rho} & \mbox{for } z = w
    \end{cases}
\end{equation*}
Then $G$ is continuous in $K \times K$ and $G(z,w) < 1$ for all $(z,w) \in K \times K$, and (\ref{eq2}) follows.
\end{proof}

The following lemma shall complete the proof of the theorem.

\begin{lemma}
The interior of the set 
$$
\Lambda := \bigcap\limits_{n = 0}^{\infty} (F_{n})^{-1}(S)
$$
is empty.
\end{lemma}

\begin{proof}
Since for any $\lambda$ we have $$f_{\lambda} (S^{+}) = \{z \in \mathbb{C}: \Im(z) > 0 \},$$ $$f_{\lambda}(S^{-}) = \{z \in \mathbb{C}: \Im (z) < 0 \}$$
and 
$$
f_{\lambda}(\mathbb{R}) = (0, \infty)
$$
it follows that
$$
\bigcap\limits_{n = 0}^{\infty} (F_{n})^{-1}(S) = \bigcap\limits_{n = 0}^{\infty} (F_{n})^{-1}(S^{+}) \cup \bigcap\limits_{n = 0}^{\infty} (F_{n})^{-1}(S^{-}) \cup \mathbb{R}.
$$
We shall prove that $ \bigcap\limits_{n = 0}^{\infty} (F_{n})^{-1}(S^{+})$ has empty interior. The case of $S^{-}$ can be done in an analogous way. \\
Let us assume the opposite, that is suppose there exists $V \subset \mathbb{C}$, a nonempty, open, connected and bounded set with
$$
V \subset \overline{V} \subset  \bigcap\limits_{n = 0}^{\infty} (F_{n})^{-1}(S^{+}).
$$
Then of course the family $(F_n |_{V})_{n = 0}^{\infty}$ is normal. Now let us fix a disk $W$ contained with its closure in $V$. Put $\delta := $ dist$(W, \partial V) > 0.$
Let $N$ be an integer large enough so that 
$$
\left(\frac{\pi}{2} \right)^N \cdot \frac{\delta}{72} > 2 \pi.
$$
Now, seeking a contradiction, assume there exists $\xi \in W$ such that for at least $N$ integers $n_1, n_2, ... , n_N \geqslant 0$ we have
$$
F_{n_i}(\xi) \in \{z \in \mathbb{C}: \Im (z) > \frac{\pi}{2} \}.
$$
Then $|(F_{n_N})'(\xi)| > \left( \frac{\pi}{2} \right)^N$, and again Bloch's Theorem implies that $F_{n_N}(W)$ contains some ball of radius $2 \pi$. Since $F_{n_N}(W)$ does not intersect the Julia set for the sequence $\{\lambda_i \}_{i = n_N}^{\infty}$, but the copies of the real line $\mathbb{R} + 2 \pi i \mathbb{Z}$ lie in this Julia set, so we arrive at a contradiction. \\
Thus we conclude that for any $z \in W$ the trajectory $F_n(z)$ visits $$\{z \in \mathbb{C}: \Im (z) > \frac{\pi}{2} \}$$ at most $N$ times. For every integer $k \geqslant 0$ let 
$$
W_k := W \cap \bigcap\limits_{n = k}^{\infty} (F_n)^{-1} (\{z \in \mathbb{C}: \Im (z) \leqslant \frac{\pi}{2} \}).
$$
Each set $W_k$ is closed in $W$, and as we have just proved
$$
W = \bigcup\limits_{k=0}^{\infty} W_k.
$$ Since $W$ is an open subset of $\mathbb{C}$ it is completely metrizable, and the Baire Category Theorem holds for it. Thus there exists $q_1 \geqslant 0$ such that 
$$
W^* := \text{Int}_{\mathbb{C}} (W_{q_1}) \neq \emptyset.
$$\\
This means that for all integers $n \geqslant q_1 \geqslant 0$, we have 
\begin{equation}\label{pozostaje}
    F_{n} (W^*) \subset \{z \in \mathbb{C}: 0 < \Im (z) < \frac{\pi}{2} \}.
\end{equation}
Consequently,
$$
F_n(W^*) \subset \{z \in \mathbb{C}: \Re (z) > 0 \}
$$
for all $n > q_1$. 
Finally note that there exists a constant $M$ (dependent on $\bar{\lambda}$) such that, if $\Re(z) > M$, $\Im(z) \in (0, \frac{\pi}{2})$, and $f_{\lambda_n} (z) \in S$ (for all $\lambda_n > \bar{\lambda} > 0$) then 

$$
\Re(f_{\lambda_n}) > \Re(z) + 1.
$$

We shall now finish the proof by excluding all possible limits of subsequences of $F_n$. Firstly, assume there is a subsequence $n_k$ such that $((F_{n_k}) |_{W^*})_{k = 1}^{\infty}$ converge to infinity. This implies that $((F_{n_k})' |_{W^*})_{k = 1}^{\infty}$ converge to infinity, which once again can be excluded by a combination of Bloch's Theorem and (\ref{pozostaje}). \\
There can also be no subseqence converging to a point in $S^+$, as all the maps $f_{\lambda_n}|_{W^*}$, $n > q_1$ expand the hyperbolic metric $\rho$.  \\
Thus let $g$ be a non-constant limit of some subsequence $(F_{n_k})_{k = 1}^{\infty}$ converging uniformly. Shrinking $W^*$ if necessary, one can assume $g(W^*)$ is contained in some compact subset $K \subset S^{+}$. Putting
$$
\widetilde{K} := \{z \in S^{+} : \rho(z,K) \leqslant 1 \},
$$
we see that there is $q_2 > q_1$ such that for every $k \geqslant q_2$

$$
F_{n_{k}}(W^*) \subset \widetilde{K}.
$$
Note that $\widetilde{K}$ has finite hyperbolic diameter, let us denote $D:= \text{diam}_{\rho} (\widetilde{K}) < \infty$. Let $z,w \in W^*$ with $z \neq w$. Then, using (\ref{eq1}) and (\ref{eq2}), we see that $\rho(z,w) \leqslant \kappa^{k-q_2} D$ for every $k \geqslant q_2$, which is a contradiction. Thus there can also be no subsequences with non-constant limits in $S^{+}$. \\
Since all limits of subsequences of $(F_{n})_{n=0}^{\infty}$ with values in $S^{+}$ have been excluded, the only possibility left is the convergence to a constant on the real line. In particular this gives the following valuable corollary:
For every $\theta >0$ there exists $n_{\theta} > 0$ such that for all $n > n_{\theta}$

\begin{equation}\label{wykgran}
F_{n} (W^*) \subset \{z \in \mathbb{C}: 0 < \Im (z) < \theta \} \cap \{z \in \mathbb{C}: 0 < \Re (z) < M \}.
\end{equation}

It is enough now to note that the above possibility is excluded since by the assumption of Proposition \ref{cora} there can be no subsequences with limits on the real line. The contradiction concludes the proof.

\end{proof}

\end{proof}

Proposition \ref{corjed} can be done analogously with the help of the following lemma (which is  Lemma 14 from \cite{ZU}):

\begin{lemma}
Let $\delta >0$ be small enough so that $1 - \delta > \frac{1}{\bar{\lambda} e}$. Then for every $\lambda \geqslant \bar{\lambda}$ and for every $z \in \mathbb{C}$ with $\cos(\Im (z)) > 1 - \delta$, we have that
$$
\Re (f_{\lambda}(z)) > \Re (z) + \bar{\lambda} e (1 - \delta).
$$
\end{lemma}

As the authors of \cite{ZU} point out, the proof is just a simple calculation. It is worth noting that in \cite{ZU} the above lemma is used to conclude that for a certain $\varepsilon > 0$ the strip
$$
\{z \in \mathbb{C}: 0 < \Im (z) < \varepsilon \}
$$
is moved to the right by a fixed amount. Should we want to apply it to prove Proposition \ref{corjed}, we would like to say that $f_{\lambda}$ pushes to the right the set
$$
\{z \in \mathbb{C}: 0 < \Im (z) < \varepsilon \} \setminus D(1, \varepsilon_2)
$$
i.e. the same strip, but without a certain (small enough) neighbourhood of $1$. This way we can exclude the possibility of any constant limits on the real line, different from $1$. The possibility of convergence to $1$ is explicitly forbidden in the assumptions of Proposition \ref{corjed}. Thus both Propositions \ref{cora}, \ref{corjed} should really be seen as corollaries from the proof of Theorem 7 in \cite{ZU}.

\bibliographystyle{amsplain}

\begin{thebibliography}{10}

\bibitem{RB} R. Brück, \textit{Connectedness and stability of Julia sets of the composition of polynomials of the form $z^2 + c_n$} Journal of the London Mathematical Society, 61, 462-470 (2000)


\bibitem{BB} R. Brück, M. Büger, S. Reitz, \textit{Random iteration of polynomials of the form $z^2 + c$: connectedness of Julia sets.} Ergodic Theory and Dynamical Systems, 19, 1221-1231 (1999)

\bibitem{Com} M. Comerford, \textit{Conjugacy and counterexample in random iteration} Pacific Journal of Mathematics, 211, 69-80 (2003)

\bibitem{FS} J. E. Fornæss , N. Sibony, \textit{Random iterations of rational functions}, Ergodic Theory and Dynamical Systems, 11, 687-708 (1991)

\bibitem{QIU} Z. Gong, W. Qiu, Y. Li, \textit{Connectedness of Julia sets for a quadratic random dynamical system}, Ergodic Theory and Dynamical Systems, 23, 1807-1815 (2003)

\bibitem {MU} V. Mayer, M. Urbański, \textit{Random dynamics of transcendental functions.} Journal d'Analyse Mathématique, 134, 201-235 (2018)

\bibitem{MUZ} V. Mayer, M. Urbański, A. Zdunik, \textit{Real analyticity for random dynamics of transcendental functions}, https://doi.org/10.1017/etds.2018.42

\bibitem{MM} M. Misiurewicz, \textit{On iterates of e^z}, Ergodic Theory and Dynamical Systems, 1, 103-106, (1981)

\bibitem {ZU} M. Urbański, A. Zdunik, \textit{Random non-hyperbolic exponential maps.
} 	arXiv:1805.08050
(2018).

\bibitem {ZU2}M. Urbański, A. Zdunik, \textit{Real analyticity of Hausdorff dimension of finer Julia sets of exponential family.} Ergodic Theory and Dynamical Systems, 24, 279-315 (2004)

\end{thebibliography}

\end{document}